\documentclass[a4paper,10pt]{amsart}
\usepackage[utf8]{inputenc}
\usepackage[T1]{fontenc}
\usepackage[english]{babel}
\usepackage{amsmath,amssymb,amsthm}
\usepackage{bbm}
\usepackage{verbatim}

\usepackage[pdftex]{color}

\usepackage[bookmarks=true,hyperindex,pdftex,colorlinks, citecolor=blue,linkcolor=blue,urlcolor=blue]{hyperref}

\usepackage{fancybox}

\let\oldtocsection=\tocsection
\let\oldtocsubsection=\tocsubsection
\let\oldtocsubsubsection=\tocsubsubsection
\renewcommand{\tocsection}[2]{\hspace{0em}\oldtocsection{#1}{#2}}
\renewcommand{\tocsubsection}[2]{\hspace{1em}\oldtocsubsection{#1}{#2}}
\renewcommand{\tocsubsubsection}[2]{\hspace{2em}\oldtocsubsubsection{#1}{#2}}

\theoremstyle{plain}
\newtheorem{Theo}{Theorem}[section]
\newtheorem{Prop}[Theo]{Proposition}
\newtheorem{Coro}[Theo]{Corollary}
\newtheorem{Lemm}[Theo]{Lemma}

\theoremstyle{definition}

\newtheorem{Exam}[Theo]{Example}
\newtheorem{Examples}[Theo]{Examples}
\newtheorem{Rema}[Theo]{Remark}

\newcommand{\conv}{\operatorname{conv}}

\newcommand{\Id}{\operatorname{Id}}

\newcommand{\eps}{\varepsilon}

\newcommand{\re}{\operatorname{Re}}

\newcommand{\e}{\operatorname{e}}

\newcommand{\N}{\mathbb{N}}

\renewcommand{\geq}{\geqslant}
\renewcommand{\leq}{\leqslant}

\numberwithin{equation}{section}

\begin{document}

\title{Numerical index and Daugavet property of operator ideals and tensor products}

\author[M.~Mart\'{\i}n]{Miguel Mart\'{\i}n}
\author[J.~Mer\'{\i}]{Javier Mer\'{\i}}
\author[A.~Quero]{Alicia Quero}
\address{Universidad de Granada \\ Facultad de Ciencias \\
Departamento de An\'{a}lisis Matem\'{a}tico \\ E-18071 Granada \\
Spain}
\email{mmartins@ugr.es, jmeri@ugr.es, aliciaquero@ugr.es}

\thanks{Research partially supported by projects PGC2018-093794-B-I00 (MCIU/AEI/FEDER, UE) and FQM-185 (Junta de Andaluc\'{\i}a/FEDER, UE). The third author is also supported by the Ph.D. scholarship FPU18/03057 (MECD)}

\thispagestyle{plain}

\date{May 22nd, 2020}

\keywords{Banach space; numerical index; numerical range; numerical radius; operator ideal; projective and injective tensor product; Daugavet property; slicely countably determined sets and operators}
\subjclass[2010]{Primary 46B04, 47A12; Secondary 46B20, 46B28, 47B07}

\maketitle

\begin{abstract}
We show that the numerical index of any operator ideal is less than or equal to the minimum of the numerical indices of the domain and the range. Further, we show that the numerical index of the ideal of compact operators or the ideal of weakly compact operators is less than or equal to the numerical index of the dual of the domain, and this result provides interesting examples. We also show that the numerical index of a projective or injective tensor product of Banach spaces is less than or equal to the numerical index of any of the factors. Finally, we show that if a projective tensor product of two Banach spaces has the Daugavet property and the unit ball of one of the factor is slicely countably determined or its dual contains a point of Fr\'{e}chet differentiability of the norm, then the other factor inherits the Daugavet property. If an injective tensor product of two Banach spaces has the Daugavet property and one of the factors contains a point of Fr\'{e}chet differentiability of the norm, then the other factor has the Daugavet property.
\end{abstract}

\section{Introduction}
The numerical index of a Banach space is a constant that relates the numerical radius and the norm of bounded linear operators on the space. It was introduced by G.~Lumer in 1968 (see \cite{D-Mc-P-W}). Let us present the needed definitions and notation. Given a Banach space $X$, we write $S_X$ and $B_X$ to denote, respectively, the unit sphere and the closed unit ball of the space. By $X^\ast$ we denote the topological dual of $X$ and $\mathcal{L}(X)$ will denote the Banach space of all bounded linear operators on $X$. The \emph{numerical range} of an operator $T\in \mathcal{L}(X)$ is the set of scalars given by
$$
V(T):=\{x^\ast(Tx)\colon x\in S_X,\,x^\ast\in S_{X^\ast}, x^\ast(x)=1\},
$$
and the \emph{numerical radius} of $T$ is then given by
$$
v(T):=\sup\{|\lambda|\colon \lambda \in V(T)\}.
$$
It is clear that the numerical radius is a seminorm on $\mathcal{L}(X)$ which is not greater than the operator norm. Very often, the numerical radius is actually an equivalent norm on $\mathcal{L}(X)$ and to quantify this fact it is used the \emph{numerical index} of the space $X$:
\begin{align*}
  n(X)& :=  \inf\{v(T)\colon T\in \mathcal{L}(X),\,\|T\|=1\} \\
   & = \max\{k\geq 0\colon k\|T\|\leq v(T)\,\forall T\in \mathcal{L}(X)\}.
\end{align*}
It is clear that $0\leq n(X)\leq 1$; the value $n(X)=1$ means that the numerical radius and the norm coincide, while $n(X)=0$ when the numerical radius is not an equivalent norm on $\mathcal{L}(X)$. We refer the reader to the expositive paper \cite{KaMaPa}, to Chapter~$1$ of the recent book \cite{SpearsBook}, and to Subsection~1.1 of the very recent paper \cite{secondnumericalindex}. Some results on numerical index which we would like to emphasize are the following. For every Banach space $X$, $n(X^\ast)\leq n(X)$ and the inequality can be strict; $n(c_0)=n(\ell_1)=n(\ell_\infty)=1$, a result which is also valid for all $L$- and $M$-spaces, the disk algebra, and $H^\infty$. The numerical index behaves differently when dealing with real or complex Banach spaces. For instance, Hilbert spaces of dimension greater than or equal to two have numerical index $0$ in the real case and $1/2$ in the complex case. In general, if $X$ is a complex Banach space, then $n(X)\geq 1/\e$ and all the values in the interval $[1/\e,1]$ are valid; for real Banach spaces, there is no restriction and all the values of the interval $[0,1]$ are possible. The numerical index of $L_p$ spaces for $1<p<\infty$, $p\neq 2$, is still unknown, but it is known that $n(L_p(\mu))>0$ in the real case for $p\neq 2$. All these results can be found in the cited papers \cite{SpearsBook,KaMaPa,secondnumericalindex}. Some recent results can be found in \cite{MeriQuero-lp}, where the exact value of some two-dimensional $\ell_p$ spaces is calculated, and in \cite{BothBaker,AksoyLewi,SainPaulBhuniaBag}, for instance. Different extensions of the concept of numerical index appear in \cite{KMMPQ} and \cite{WandHuangTan}.

There is a property somehow related to the numerical index called Daugavet property. A Banach space $X$ has the \emph{Daugavet property} \cite{KSSW} if the norm equality
\begin{equation}\label{DE}\tag{DE}
\|\Id+T\|=1+\|T\|
\end{equation}
holds for all rank-one operators $T\in \mathcal{L}(X)$ and, in this case, the same happens for all weakly compact operators on $X$. Examples of Banach spaces satisfying this property are $L_1(\mu,Y)$ when the positive measure $\mu$ is atomless and $Y$ is arbitrary, $C(K,Y)$ when the compact space $K$ is perfect and $Y$ is arbitrary, or the disk algebra. Let us say that there is a relation between the Daugavet property and the numerical range of operators: an operator $T$ satisfies \eqref{DE} if and only if $\sup\re V(T)=\|T\|$ (see \cite{D-Mc-P-W} for instance). Classical references for Daugavet property include \cite{KSSW,Shvydkoy,Werner97}. For very recent results, we refer the reader to \cite{BrachSanchezWerner, RuedaTradaceteVillanueva}, for instance.

To state the results of the paper, we need to introduce some definitions and notation. Given Banach spaces $X$ and $Y$, we write $\mathcal{L}(X,Y)$, $\mathcal{K}(X,Y)$, $\mathcal{W}(X,Y)$, and $\mathcal{A}(X,Y)$ to denote, respectively, the space of (bounded linear) operators, compact operators, weakly compact operators, and approximable operators (i.e.\ norm limits of finite rank operators), all of them endowed with the operator norm. Finally, we consider the space of all nuclear operators: an operator $T\colon X\longrightarrow Y$ between Banach spaces is called \textit{nuclear} if there exist $x_n^\ast\in X$ and $y_n\in Y$ for every $n\in\N$ such that $\sum_{n=1}^\infty\|x_n^\ast\|\,\|y_n\|<\infty$ and
$$Tx=\sum_{n=1}^\infty x_n^\ast(x)y_n \qquad (x\in X).$$
The space of all nuclear operators, denoted by $\mathcal{N}(X,Y)$, is a Banach space endowed with the norm
$$N(T)=\inf\left\{\sum_{n=1}^\infty \|x_n^\ast\|\,\|y_n\|\colon Tx=\sum_{n=1}^\infty x_n^\ast(x)y_n \right\},$$
where the infimum is taken over all the representations of $T$ as above. The \textit{projective tensor product} of $X$ and $Y$, denoted by $X\hat{\otimes}_\pi Y$, is the completion of $X\otimes Y$ under the norm given by
$$
\|u\|_\pi=\inf\left\{\sum_{i=1}^n\|x_i\|\,\|y_i\|\colon u=\sum_{i=1}^n x_i\otimes y_i \right\},
$$
where the infimum is taken over all the representations of $u=\sum_{i=1}^n x_i\otimes y_i$. It follows from the definition that $B_{X\hat{\otimes}_\pi Y}=\overline{\conv}(B_X\otimes B_Y)$. The projective tensor product of two operators $S\in \mathcal{L}(X,W)$ and $T\in \mathcal{L}(Y,Z)$ between Banach spaces, denoted by $S\otimes_\pi T$, is the unique operator between $X\hat{\otimes}_\pi Y$ and $W\hat{\otimes}_\pi Z$ such that $(S\otimes_\pi T)(x\otimes y)=Sx\otimes Ty$ for every $x\in X$ and $y\in Y$, which also satisfies that $\|S\otimes_\pi T\|=\|S\|\,\|T\|$. The \textit{injective tensor product} of $X$ and $Y$, denoted by $X\hat{\otimes}_\eps Y$, is the completion of $X\otimes Y$ under the norm given by
$$
\|u\|_\eps=\sup\left\{ \left|\sum_{i=1}^n x^\ast(x_i)y^\ast(y_i)\right|\colon x^\ast\in B_{X^\ast},\, y^\ast\in B_{Y^\ast} \right\},
$$
where $\sum_{i=1}^n x_i\otimes y_i$ is any representation of $u$. The injective tensor product of two operators $S\in \mathcal{L}(X,W)$ and $T\in \mathcal{L}(Y,Z)$ between Banach spaces, denoted by $S\otimes_\eps T$, is the unique operator between $X\hat{\otimes}_\eps Y$ and $W\hat{\otimes}_\eps Z$ such that $(S\otimes_\eps T)(x\otimes y)=Sx\otimes Ty$ for every $x\in X$ and $y\in Y$, which also satisfies that $\|S\otimes_\eps T\|=\|S\|\,\|T\|$. We refer the reader to \cite{DefantFloret} and \cite{Ryan} for more information and background about ideals of operators and tensor products of Banach spaces.

For ideals of operators, we show in Section~\ref{section:Numindex-ideals} that for every operator ideal $\mathcal{Z}$ of $\mathcal{L}(X,Y)$ endowed with the operator norm we have that $n(\mathcal{Z})\leq \min\{n(X),n(Y)$. In the case of compact and weakly compact operators, we may improve this inequality to
$$
n(\mathcal{K}(X,Y))\leq \min\{n(X^\ast),n(Y)\},\qquad n(\mathcal{W}(X,Y))\leq \min\{n(X^\ast),n(Y)\}.
$$
This result allows us to present some interesting examples as the existence of a real Banach space $X$ such that $n(X)=1$ while $n(\mathcal{K}(X,Y))=n(\mathcal{W}(X,Y))=0$ for every Banach space $Y$. In particular, $n(X)=1$ while $n(\mathcal{K}(X,X))=n(\mathcal{W}(X,X))=0$.

For tensor products of Banach spaces, we prove in Section~\ref{section:tensorproduct-numindex} that the numerical indices of $X\hat{\otimes}_\pi Y$ and $X\hat{\otimes}_\eps Y$ are less than or equal to the minimum of $n(X)$ and $n(Y)$. As a consequence, and just using representation theorems, we get some consequences for the space of approximable operators and for the space of nuclear operators:
$$
n(\mathcal{A}(X,Y))\leq \min\{n(X^\ast),n(Y)\}
$$
and, in the case where $X^\ast$ or $Y$ has the approximation property,
$$
n(\mathcal{N}(X,Y))\leq \min\{n(X^\ast),n(Y)\}.
$$

Finally, we study in Section~\ref{section:Daugavet} the Daugavet property of tensor products of Banach spaces. We show that when $X\hat{\otimes}_\pi Y$ has the Daugavet property and $B_Y$ is a slicely countably determined set (see the definition at the beginning of the section), then $X$ has the Daugavet property. We also provide with the analogous result in the case where the space $Y^\ast$ has a point of Fr\'{e}chet differentiability of the norm. For injective tensor products, we do not know if the result with the hypothesis of slicely countably determined unit ball is true or not, but there is a positive result when the space $Y$ has a point of Fr\'{e}chet differentiability of the norm.

\section{Numerical index of some operator ideals of $\mathcal{L}(X,Y)$}\label{section:Numindex-ideals}
Given two Banach spaces $X$ and $Y$, we first study the relationship between the numerical index of some operator ideals of $\mathcal{L}(X,Y)$ and the numerical indices of the spaces $X$ and $Y$. Recall that an \emph{operator ideal} in $\mathcal{L}(X,Y)$ is a closed subspace $\mathcal{Z}$ containing all finite-rank operators and satisfying that $ATB\in \mathcal{Z}$ whenever $A\in \mathcal{L}(Y)$, $T\in \mathcal{Z}$, and $B\in \mathcal{L}(X)$.

\begin{Prop}\label{num-idex-Ldex}
Let $X$, $Y$ be Banach spaces, then $n\big(\mathcal{L}(X,Y)\big)\leq \min\{n(X),n(Y)\}$. Moreover, the same happens to every operator ideal $\mathcal{Z}\leq \mathcal{L}(X,Y)$ endowed with the operator norm, that is, $n(\mathcal{Z})\leq \min\{n(X),n(Y)\}$.
\end{Prop}

To give the proof of the proposition, we need the following lemma which is well known and can be deduced, for instance, from \cite[Corollary~2.1.2]{Cabrera-Rodriguez}.

\begin{Lemm}\label{lemma:isometric-inclusion-num-index}
Let $X_1$, $X_2$ be Banach spaces and suppose that there is an isometric embedding $\Phi\colon \mathcal{L}(X_1)\longrightarrow \mathcal{L}(X_2)$ satisfying $\Phi(\Id_{X_1})=\Id_{X_2}$. Then, $n(X_2)\leq n(X_1)$.
\end{Lemm}

\begin{proof}[Proof of Proposition~\ref{num-idex-Ldex}]
	We first show that $n\big(\mathcal{L}(X,Y)\big)\leq n(X)$. Fixed $J\in\mathcal{L}(X)$, we define the map $\Phi_J \colon \mathcal{L}(X,Y) \longrightarrow \mathcal{L}(X,Y)$ by $\Phi_J(T)=T\circ J$ for every $T\in\mathcal{L}(X,Y)$ and observe that $\|\Phi_J\|=\|J\|$. Indeed, the inequality $\|\Phi_J\|\leq \|J\|$ is evident. To prove the reverse one, given $\eps>0$, we find $x_\eps\in S_X$ satisfying $\|Jx_\eps\|>\|J\|-\eps$ and then we take $x^\ast_\eps\in S_{X^\ast}$ such that $x^\ast_\eps(Jx_\eps)=\|Jx_\eps\|>\|J\|-\eps$. We fix $y_0\in S_Y$ and define the rank-one operator $T_\eps\in \mathcal{L}(X,Y)$ by $T_\eps(x)=x^\ast_\eps(x)y_0$ for every $x\in X$, which satisfies $\|T_\eps\|=1$ and
	$$
	\|\Phi_J(T_\eps)\|=\|T_\eps\circ J\|\geq\|[T_\eps\circ J](x_\eps)\|=\|x^\ast_\eps(Jx_\eps)y_0\|>\|J\|-\eps.
	$$
	Therefore $\|\Phi_J\|\geq\|J\|$, and hence the mapping $J\longmapsto \Phi_J$ is an isometric embedding from $\mathcal{L}(X)$ to $\mathcal{L}\big(\mathcal{L}(X,Y)\big)$ carrying $\Id_X$ to $\Id_{\mathcal{L}(X,Y)}$, so the inequality $n\big(\mathcal{L}(X,Y)\big)\leq n(X)$ follows by Lemma~\ref{lemma:isometric-inclusion-num-index}. The inequality $n\big(\mathcal{L}(X,Y)\big)\leq n(Y)$ can be proved analogously, using $\Psi_S(T)=S\circ T$ instead of $\Phi_J$.
	
	To prove the moreover part it suffices to observe that if $T\in \mathcal{Z}\subset \mathcal{L}(X,Y)$ and $J\in \mathcal{L}(X)$, then $\Phi_J(T)=T\circ J$ belongs to $\mathcal{Z}$ for every $T\in \mathcal{Z}$, as $\mathcal{Z}$ is an operator ideal. So the map $J\longmapsto \Phi_J$ is an isometric embedding from $\mathcal{L}(X)$ to $\mathcal{L}(\mathcal{Z})$ carrying $\Id_X$ to $\Id_\mathcal{Z}$, and the result follows again by Lemma~\ref{lemma:isometric-inclusion-num-index}. For the inequality involving $n(Y)$, the argument is analogous, considering now that $\Psi_S(T)=S\circ T\in \mathcal{Z}$ for every $T\in \mathcal{Z}$ and so the map $S\longmapsto \Psi_S$ is an isometric embedding from $\mathcal{L}(Y)$ to $\mathcal{L}(\mathcal{Z})$ carrying $\Id_Y$ to $\Id_{\mathcal{Z}}$.
	\end{proof}

We can get a stronger result for the numerical indices of $\mathcal{K}(X,Y)$ and $\mathcal{W}(X, Y)$. To do so, we recall that $\mathcal{K}_{w^\ast}(X^\ast,Y)$ denotes the space of compact operators that are weak$^\ast$-weakly continuous from $X^\ast$ into $Y$ endowed with the usual operator norm. This space was originally introduced by L.~Schwartz \cite{Schwartz} as the $\eps$-product of the spaces $X$ and $Y$. It is well-known that $\mathcal{K}_{w^\ast}(X^{\ast},Y)\equiv \mathcal{K}_{w^\ast}(Y^{\ast},X)$ and that $\mathcal{K}(X,Y)$ can be identified with $\mathcal{K}_{w^\ast}(X^{\ast\ast},Y)$ using the mapping $T\longmapsto T^{\ast\ast}$. Analogously, $\mathcal{L}_{w^\ast}(X^\ast,Y)$ denotes the space of operators that are weak$^\ast$-weakly continuous from $X^\ast$ into Y. Finally, we recall that $\mathcal{W}(X,Y)$ can be identified with $\mathcal{L}_{w^\ast}(X^{\ast\ast},Y)$. We refer the reader to \cite{Ruess, Schwartz} for background on this type of spaces.

\begin{Theo}\label{thm-num-index-compact-wcompact}
Let $X$, $Y$ be Banach spaces, then the following hold:
\begin{itemize}
	\item[(a)] $n\big(\mathcal{L}_{w^\ast}(X^\ast,Y)\big)\leq \min\{n(X),n(Y)\}$.
	\item[(b)] $n\big(\mathcal{K}_{w^\ast}(X^\ast,Y)\big)\leq \min\{n(X),n(Y)\}$.
	\item[(c)] $n\big(\mathcal{W}(X,Y)\big)\leq \min\{n(X^\ast),n(Y)\}$.
	\item[(d)] $n\big(\mathcal{K}(X,Y)\big)\leq \min\{n(X^\ast),n(Y)\}$.
\end{itemize}
\end{Theo}

\begin{proof}
	To prove (a), for $J\in\mathcal{L}(X)$ we define the operator $\Psi_J\colon \mathcal{L}_{w^\ast}(X^\ast,Y)\longrightarrow \mathcal{L}_{w^\ast}(X^\ast,Y)$ given by $\Psi_J(T)=T\circ J^\ast$ for every $T\in\mathcal{L}_{w^\ast}(X^\ast,Y)$. Observe that it is well-defined because $J^\ast$ is weak$^\ast$-weak$^\ast$ continuous. Moreover, reasoning as in the proof of Proposition~\ref{num-idex-Ldex} we get $\|\Psi_J\|=\|J\|$. Therefore, the mapping $J\longmapsto \Psi_J$ is an isometric embedding from $\mathcal{L}(X)$ to $\mathcal{L}\big(\mathcal{L}_{w^\ast}(X^\ast,Y)\big)$ carrying $\Id_X$ to $\Id_{\mathcal{L}_{w^\ast}(X^\ast,Y)}$ so the inequality $n\big(\mathcal{L}_{w^\ast}(X^\ast,Y)\big)\leq n(X)$ follows from Lemma~\ref{lemma:isometric-inclusion-num-index}.
	
	The proof of $n\big(\mathcal{L}_{w^\ast}(X^\ast,Y)\big)\leq n(Y)$ can be done analogously. Indeed, for $J\in \mathcal{L}(Y)$ define the operator $\Gamma_J\colon \mathcal{L}_{w^\ast}(X^\ast,Y)\longrightarrow \mathcal{L}_{w^\ast}(X^\ast,Y)$ given by $\Gamma_J(T)=J\circ T$ for every $T\in\mathcal{L}_{w^\ast}(X^\ast,Y)$, which is well-defined because $J$ is weak-weak continuous. As before, it is easy to check that  $\|\Gamma_J\|=\|J\|$ so the mapping $J\longmapsto \Gamma_J$ is an isometric embedding from $\mathcal{L}(Y)$ to $\mathcal{L}\big(\mathcal{L}_{w^\ast}(X^\ast,Y)\big)$ carrying $\Id_Y$ to $\Id_{\mathcal{L}_{w^\ast}(X^\ast,Y)}$. Therefore, the inequality $n\big(\mathcal{L}_{w^\ast}(X^\ast,Y)\big)\leq n(Y)$ follows from Lemma~\ref{lemma:isometric-inclusion-num-index}.	
	
	Let us prove (b). To show that $n\big(\mathcal{K}_{w^\ast}(X^\ast,Y)\big)\leq n(X)$ it suffices to observe that ${\Psi_J}|_{\mathcal{K}_{w^\ast}(X^\ast,Y)}$, the restriction of $\Psi_J$ to $\mathcal{K}_{w^\ast}(X^\ast,Y)$, lies in $\mathcal{L}\big(\mathcal{K}_{w^\ast}(X^\ast,Y)\big)$ and satisfies $\left\|{\Psi_J}_{|\mathcal{K}_{w^\ast}(X^\ast,Y)}\right\|=\|J\|$. Therefore, the mapping $J\longmapsto {\Psi_J}_{|\mathcal{K}_{w^\ast}(X^\ast,Y)}$ is an isometric embedding from $\mathcal{L}(Y)$ to $\mathcal{L}\big(\mathcal{K}_{w^\ast}(X^\ast,Y)\big)$ carrying $\Id_X$ to $\Id_{\mathcal{K}_{w^\ast}(X^\ast,Y)}$ and Lemma~\ref{lemma:isometric-inclusion-num-index} gives the result.
	
	To prove $n\big(\mathcal{K}_{w^\ast}(X^\ast,Y)\big)\leq n(Y)$ one can proceed as in (a) or use what we just proved and the identification $\mathcal{K}_{w^\ast}(X^{\ast},Y)\equiv \mathcal{K}_{w^\ast}(Y^{\ast},X)$.
	
	(c) follows from (a) using the identification $\mathcal{W}(X,Y)\equiv\mathcal{L}_{w^\ast}(X^{\ast\ast},Y)$.
	
	(d) follows from (b) using the identification $\mathcal{K}(X,Y)\equiv\mathcal{K}_{w^\ast}(X^{\ast\ast},Y)$.
\end{proof}

As a consequence of \cite[Examples 3.3]{BoykoKadetsMartinWerner} and Theorem~\ref{thm-num-index-compact-wcompact} we have the following interesting examples.

\begin{Examples}\label{examples:nX=1,nK=0} \quad
	\begin{itemize}
		\item[(a)] {\slshape There exists a real Banach space $X$ with $n(X)=1$ and $n\big(\mathcal{K}(X,Y)\big)=n\big(\mathcal{W}(X,Y)\big)=0$ for every Banach space $Y$. In particular, $n(X)=1$ and $n(\mathcal{K}(X,X))=n(\mathcal{W}(X,X))=0$.} Indeed, the real space $X$ given in \cite[Examples 3.3.a]{BoykoKadetsMartinWerner} satisfies $n(X)=1$ and $n(X^\ast)=0$ so $n\big(\mathcal{K}(X,Y)\big)=n\big(\mathcal{W}(X,Y)\big)=0$ for every $Y$ by Theorem~\ref{thm-num-index-compact-wcompact}.
		\item[(b)] {\slshape There exists a complex Banach space $X$ with $n(X)=1$ and $n\big(\mathcal{K}(X,Y)\big)=n\big(\mathcal{W}(X,Y)\big)=1/\e$ for every Banach space $Y$. In particular, $n(X)=1$ and $n(\mathcal{K}(X,X))=n(\mathcal{W}(X,X))=1/\e$.} The complex space $X$ given in \cite[Examples 3.3.b]{BoykoKadetsMartinWerner} satisfies $n(X)=1$ and $n(X^\ast)=1/\e$, so it works by Theorem~\ref{thm-num-index-compact-wcompact} and the fact that every complex Banach space has numerical index less than or equal to $1/\e$.
	\end{itemize}
\end{Examples}

To obtain the analogue of Theorem~\ref{thm-num-index-compact-wcompact} for the numerical index of the space of approximable operators and also to get an analogous result for nuclear operators, we will use their representation as suitable tensor products in the next section.

We emphasize a consequence of the results for the case when the ideal spaces have numerical index one.

\begin{Coro}\label{corollary-L(X,Y)=1}
Let $X$, $Y$ be Banach spaces.
\begin{enumerate}
  \item If $n(\mathcal{L}(X,Y))=1$, then $n(X)=n(Y)=1$.
  \item If $n(\mathcal{K}(X,Y))=1$, then $n(X^\ast)=n(Y)=1$.
  \item If $n(\mathcal{W}(X,Y))=1$, then $n(X^\ast)=n(Y)=1$.
\end{enumerate}
\end{Coro}

One may wonder whether the inequalities obtained for the numerical indices of operator ideals are equalities in general. The following example shows that this is not the case, even for finite-dimensional spaces.

\begin{Exam}
{\slshape There exist finite-dimensional Banach spaces $X$ and $Y$ with $n(X^\ast)=n(Y)=1$ and $n\big(\mathcal{L}(X,Y)\big)=n\big(\mathcal{F}(X,Y)\big)=n\big(\mathcal{K}(X,Y)\big)=n\big(\mathcal{W}(X,Y)\big)<1$.} Indeed, consider $X=\ell_\infty^4$ and $Y=\ell_1^4$, which have numerical index $1$, and observe that $n\big(\mathcal{L}(X,Y)\big)<1$ by \cite[Proposition 2.4, Lemma 3.2]{Lima}.
\end{Exam}

However there are cases in which the equality holds for the spaces of compact and weakly compact operators.

\begin{Rema}
{\slshape Let $K$ be a compact Hausdorff space, and let $X$ be a Banach space. Then, $$n\big(\mathcal{K}(X,C(K))\big)=n\big(\mathcal{W}(X,C(K))\big)=n(X^\ast).$$} Indeed, the space $\mathcal{K}(X,C(K))$ can be identified with $C(K,X^\ast)$ (see \cite[Theorem VI.7.1]{DunfordSchwartz}) and we have $n\big(C(K,X^\ast)\big)=n(X^\ast)$ by \cite[Theorem~5]{MartinPaya}. The equality $n\big(\mathcal{W}(X,C(K))\big)=n(X^\ast)$ holds by \cite[Corollary~3]{LopezMartinMeri}.
\end{Rema}

In the next result we give other conditions for which the equality is satisfied for the space of compact operators.

\begin{Prop}
Let $X$ be a Banach space such that $n(X^{\ast\ast\ast})=1$ and let $Z$ be an isometric predual of $\ell_1$. Then the space $\mathcal{K}(X,Z)^{\ast\ast}$ has numerical index one. Therefore, so do $\mathcal{K}(X,Z)^\ast$ and $\mathcal{K}(X,Z)$. In particular, $n\big(\mathcal{K}(c_0)\big)=n\big(\mathcal{K}(c_0)^\ast\big)=n\big(\mathcal{K}(c_0)^{\ast\ast}\big)=n\big(\mathcal{L}(\ell_\infty)\big)=1$ and $n\big(\mathcal{K}(\ell_1,c_0)^{\ast\ast}\big)=1$.
\end{Prop}
\begin{proof}
	Since $Z$ has the approximation property, $\mathcal{K}(X,Z)\equiv X^\ast\hat{\otimes}_\eps Z$. Since $Z^\ast$ has the approximation property and the Radon-Nikod\'ym property, we can apply \cite[Theorem 16.6]{DefantFloret} to obtain $\mathcal{K}(X,Z)^\ast\equiv(X^\ast\hat{\otimes}_\eps Z)^\ast \equiv X^{\ast\ast}\hat{\otimes}_\pi\ell_1$. Therefore, $\mathcal{K}(X,Z)^{\ast\ast}\equiv(X^{\ast\ast}\hat{\otimes}_\pi\ell_1)^\ast\equiv\mathcal{L}(X^{\ast\ast},\ell_\infty)$. Now, by using the identification between $\ell_\infty$ and $C(\beta \N)$, where $\beta\N$ is the Stone--\v Cech compactification of $\N$, and the one between $C_{w^\ast}(\beta\N,X^{\ast\ast\ast})$ and $\mathcal{L}\big(X^{\ast\ast},C(\beta\N)\big)$ (see \cite[Theorem VI.7.1]{DunfordSchwartz}), we obtain that
	$$ n\big(\mathcal{L}(X^{\ast\ast},\ell_\infty)\big)=n\big(C_{w^\ast}(\beta\N,X^{\ast\ast\ast})\big)\geq n(X^{\ast\ast\ast})=1,$$
	where the inequality is given by \cite[Proposition 7]{LopezMartinMeri}. Then $n\big(\mathcal{K}(X,Z)^{\ast\ast}\big)=1$ as desired. The other statements follow straightforwardly.
\end{proof}

\section{Numerical index of tensor products}\label{section:tensorproduct-numindex}
Our goal here is to study the numerical index of projective and injective tensor products of Banach spaces. It is known that $n(X\hat{\otimes}_\eps Y)$ and $n(X\hat{\otimes}_\pi Y)$ cannot be computed as a function of $n(X)$ and $n(Y)$. Indeed, it is shown in \cite[Example~10]{MartinPaya} that there exist Banach spaces $X$ and $Y$ with $n(X)=n(Y)=1$ and such that $n(X\hat{\otimes}_\eps X)<1$, $n(Y\hat{\otimes}_\pi Y)<1$, and $n(X\hat{\otimes}_\pi X)=n(Y\hat{\otimes}_\eps Y)=1$. Therefore, our results will be inequalities, as in the previous section.

Our first result on tensor products follows immediately by Proposition~\ref{num-idex-Ldex} and the identifications $(X\hat{\otimes}_\pi Y)^\ast\equiv\mathcal{L}(X,Y^\ast)\equiv\mathcal{L}(Y,X^\ast)$ (see \cite[Proposition~3.2]{DefantFloret}, for instance).

\begin{Coro}\label{cor-num-index-dual-tensor}
	Let $X$, $Y$ be Banach spaces. Then $n\big((X\hat{\otimes}_\pi Y)^\ast\big)\leq \min\{n(X^\ast),n(Y^\ast)\}$.
\end{Coro}	

Our main result in this section is the following pair of inequalities.

\begin{Theo}\label{num-index-tensor}
	Let $X$, $Y$ be Banach spaces. Then the following hold:
	\begin{itemize}
		\item[(a)] $n(X\hat{\otimes}_\pi Y)\leq \min\{n(X),n(Y)\}$,
		\item[(b)] $n(X\hat{\otimes}_\eps Y)\leq \min\{n(X),n(Y)\}$.
	\end{itemize}
\end{Theo}

We introduce some notation in order to present an interesting tool to calculate numerical radii which we will use in the proof of the theorem. Given a Banach space $X$, $\delta>0$, and $T\in \mathcal{L}(X)$, we write
$$
v_\delta(T):=\sup\bigl\{|x^*(Tx)|\colon x\in B_X,\,x^*\in B_{X^\ast},\,\re x^*(x)>1-\delta\bigr\}.
$$

\begin{Lemm}[\mbox{\rm \cite[Lemma~3.4]{KMMPQ}}]\label{lemma-vdelta}
Let $X$ be a Banach space. For $T\in \mathcal{L}(X)$, we have that
$$
v(T)=\inf_{\delta>0}v_\delta(T).
$$
Moreover, if $A\subset B_X$ satisfies that $\overline{\conv}(A)=B_X$ and $B\subset B_{X^\ast}$ satisfies that $\overline{\conv}^{w^\ast}(B)=B_{X^\ast}$, then the same equality holds if we replace $B_X$ and $B_{X^\ast}$ by $A$ and $B$ respectively in the definition of $v_\delta(T)$, that is,
$$
v(T)=\inf_{\delta>0} \sup\bigl\{|x^*(Tx)|\colon x\in A,\,x^*\in B,\,\re x^*(x)>1-\delta\bigr\}.
$$
\end{Lemm}

\begin{proof}[Proof of the Theorem~\ref{num-index-tensor}]
	(a). We prove first $n(X\hat{\otimes}_\pi Y)\leq n(X)$. Given $S\in \mathcal{L}(X)$ with $\|S\|=1$, we consider the operator $T=S\otimes_\pi \Id_Y\in \mathcal{L}(X\hat{\otimes}_\pi Y)$ which satisfies that $\|T\|=\|S\| \|\Id_Y\|=1$. Since $B_{X\hat{\otimes}_\pi Y}=\overline{\conv}\left(B_X\otimes B_Y\right)$ and $(X\hat{\otimes}_\pi Y)^\ast=\mathcal{L}(Y,X^\ast)$, by Lemma~\ref{lemma-vdelta} we can estimate the numerical radius of $T$ as
	$$
	v(T)=\inf_{\delta>0} \tilde{v}_\delta(T),
	$$
	where for $\delta>0$,
	$$
	\tilde{v}_\delta(T):= \sup \left\{\left|\langle\Phi,Tz\rangle\right|\colon z\in B_X\otimes B_Y,\,\Phi\in B_{\mathcal{L}(Y,X^\ast)},\, \re\langle\Phi,z\rangle>1-\delta \right\}.
	$$
	Fixed $\delta>0$, we claim that $\tilde{v}_\delta(T)\leq v_\delta(S)$. Indeed, fix $z=x\otimes y\in B_X\otimes B_Y$ and $\Phi\in B_{\mathcal{L}(Y,X^\ast)}$ such that $\re\langle\Phi,z\rangle=\re\langle \Phi(y),x\rangle>1-\delta$, define $x^\ast=\Phi(y)\in B_{X^\ast}$, and observe that $\re x^\ast(x)=\re\langle\Phi,z\rangle>1-\delta$. Then,
	$$
	\left|\langle\Phi,Tz\rangle\right|=\left|\langle\Phi,Sx\otimes y\rangle\right|=\left|\langle\Phi(y),Sx\rangle\right|=\left|x^\ast(Sx)\right|\leq v_\delta(S)
	$$
	which gives $\tilde{v}_\delta(T)\leq v_\delta(S)$. Then we get that $v(T)\leq v(S)$ and, as $\|T\|=\|S\|=1$, we deduce that $n(X\hat{\otimes}_\pi Y)\leq n(X)$. By repeating this process using this time the identification $(X\hat{\otimes}_\pi Y)^\ast\equiv\mathcal{L}(X,Y^\ast)$, we also obtain that $n(X\hat{\otimes}_\pi Y)\leq n(Y)$.
	
	(b). We prove $n(X\hat{\otimes}_\eps Y)\leq n(X)$. Given $S\in \mathcal{L}(X)$ with $\|S\|=1$, we consider $T=S\otimes_\eps \Id_Y\in \mathcal{L}(X\hat{\otimes}_\eps Y)$ which satisfies that $\|T\|=\|S\| \|\Id_Y\|=1$. Since $B_{(X\hat{\otimes}_\eps Y)^\ast}= \overline{\conv}^{w^\ast}(B_{X^\ast}\otimes B_{Y^\ast})$ and
	$$
	B_{X\hat{\otimes}_\eps Y}=\overline{\{z\in X\otimes Y\colon \|z\|_\eps\leq 1\}},
	$$
	we use the following to estimate the numerical radius of $T$ (again by by Lemma~\ref{lemma-vdelta}):
	$$
	v(T)=\inf_{\delta>0} \bar{v}_\delta(T)
	$$
	where
	$$
	\bar{v}_\delta:=\sup \left\{|z^\ast(Tz)| \colon z^\ast\in B_{X^\ast}\otimes B_{Y^\ast},\, z\in X\otimes Y \text{ with } \|z\|_\eps\leq 1,\, \re z^\ast(z)>1-\delta \right\}.
	$$
	Given $\delta>0$, we claim that $\bar{v}_\delta(T)\leq v_\delta(S)$. Indeed, fixed $z=\sum_{i=1}^{n} x_i\otimes y_i \in X\otimes Y$ with $\|z\|_\eps\leq 1$ and $z^\ast=x^\ast_0\otimes y^\ast_0\in B_{X^\ast}\otimes B_{Y^\ast}$ with $\re z^\ast(z)=\re \sum_{i=1}^{n} x^\ast_0(x_i)y^\ast_0(y_i)>1-\delta$, we consider $x=\sum_{i=1}^{n} y^\ast_0(y_i)x_i \in B_X$ which satisfies
	$$
	\|x\|=\left\|\sum_{i=1}^{n} y^\ast_0(y_i)x_i\right\|\leq \sup \left\{ \left|\sum_{i=1}^{n} y^\ast_0(y_i)x^*(x_i)\right| \colon x^\ast\in B_{X^\ast} \right\}\leq\|z\|_\eps
	$$
	and $\re x^\ast_0(x)=\re \sum_{i=1}^{n} x^\ast_0(x_i)y^\ast_0(y_i)>1-\delta$. Hence we can write
	$$
	|z^\ast(Tz)|=\left|\langle x^\ast_0\otimes y^\ast_0, \sum_{i=1}^n Sx_i\otimes y_i\rangle \right|=\left|\sum_{i=1}^{n} x^\ast_0(Sx_i)y^\ast_0(y_i)\right|=|x^\ast_0(Sx)|\leq v_\delta (S).
	$$
	Then, we deduce that $\bar{v}_\delta(T)\leq v_\delta(S)$ as claimed. From this, we get that $v(S)\geq v(T)\geq n(X\hat{\otimes}_\eps Y)$. Therefore $n(X\hat{\otimes}_\eps Y)\leq n(X)$. The inequality $n(X\hat{\otimes}_\eps Y)\leq n(Y)$ follows by symmetry. 	
\end{proof}

Let us observe that it is not possible to improve Theorem~\ref{num-index-tensor} to get the numerical index of the dual of the factors in the right-hand side.

\begin{Exam}
{\slshape Let $X_1=C[0,1]$, $X_2=L_1[0,1]$ and let $Y$ be a Banach space with $n(Y)=1$ and $n(Y^\ast)<1$ (use \cite[Examples~3.3]{BoykoKadetsMartinWerner} for instance). Then, $X_1\hat{\otimes}_\eps Y\equiv C([0,1],Y)$, so $n(X_1\hat{\otimes}_\eps Y)=1$ by \cite[Theorem~5]{MartinPaya}, while $n(Y^\ast)<1$. On the other hand, $X_2\hat{\otimes}_\pi Y \equiv L_1([0,1],Y)$, so $n(X_1\hat{\otimes}_\pi Y)=1$ by \cite[Theorem~8]{MartinPaya}, while $n(Y^\ast)<1$.}
\end{Exam}

Nevertheless, the next inequality for the numerical index of the dual of an injective tensor product holds.

\begin{Coro}
	Let $X$, $Y$ be Banach spaces. If $X^\ast$ or $Y^\ast$ has the approximation property and $X$ or $Y$ has the Radon-Nikod\'ym property, then
	$$
	n\big((X\hat{\otimes}_\eps Y)^\ast\big)\leq \min\left\{n(X^\ast),n(Y^\ast)\right\}.
	$$
\end{Coro}

\begin{proof}
	The result is an immediate consequence of Theorem~\ref{num-index-tensor} as	
the identification $(X\hat{\otimes}_\eps Y)^\ast\equiv X^\ast\hat{\otimes}_\pi Y^\ast$ holds under the hypotheses (see \cite[Theorem 16.6]{DefantFloret}).
\end{proof}	

The next consequence is an inequality for the numerical index of spaces of approximable operators similar to the one given in Theorem~\ref{thm-num-index-compact-wcompact} for compact and weakly compact operators.

\begin{Coro}
Let $X$, $Y$ be Banach spaces. Then
$$
n\big(\mathcal{A}(X,Y)\big)\leq \min\left\{n(X^\ast),n(Y)\right\}.
$$
\end{Coro}

\begin{proof}
	It follows from Theorem \ref{num-index-tensor}.b as $\mathcal{A}(X,Y)\equiv X^\ast\hat{\otimes}_\eps Y$ (see \cite[Examples~4.2]{DefantFloret}).
\end{proof}

For the space of nuclear operators we may also give some interesting inequalities.

\begin{Coro}
Let $X$, $Y$ be Banach spaces. If either $X^\ast$ or $Y$ has the approximation property, then the following hold:
\begin{itemize}
	\item[(a)] $n\big(\mathcal{N}(X,Y)\big)\leq\min\{n(X^\ast),n(Y)\}$.
	\item[(b)] $n\big(\mathcal{N}(X,Y)^\ast\big)\leq\min\{n(X^{\ast\ast}),n(Y^\ast)\}$.
\end{itemize}
\end{Coro}

\begin{proof}
(a).	Since $X^\ast$ or $Y$ has the approximation property, we have that $\mathcal{N}(X,Y)\equiv X^\ast\hat{\otimes}_\pi Y$ (see \cite[Corollary~5.7.1]{DefantFloret}) and the result follows from Theorem \ref{num-index-tensor}.a.

(b). Corollary~\ref{cor-num-index-dual-tensor} gives the result using the equality $\mathcal{N}(X,Y)^\ast= \big(X^\ast\hat{\otimes}_\pi Y\big)^\ast$.
\end{proof}

Finally, we may give a result analogous to Corollary~\ref{corollary-L(X,Y)=1} for the results of this section.

\begin{Coro}\label{corollary-indextensorproduct=1}
Let $X$, $Y$ be Banach spaces.
\begin{enumerate}
  \item If $n\big((X\hat{\otimes}_\pi Y)^\ast\big)=1$, then $n(X^\ast)=n(Y^\ast)=1$.
  \item If $n(X \hat{\otimes}_\eps Y)=1$, then $n(X)=n(Y)=1$.
  \item If $n(X \hat{\otimes}_\pi Y)=1$, then $n(X)=n(Y)=1$.
  \item If $n\big(\mathcal{A}(X,Y)\big)=1$, then $n(X^\ast)=n(Y)=1$.
  \item If $n\big((X\hat{\otimes}_\eps Y)^\ast\big)=1$, then $n(X^\ast)=n(Y^\ast)=1$.
  \item If $n\big(\mathcal{N}(X,Y)\big)=1$, then $n(X^\ast)=n(Y)=1$.
  \item If $n\big(\mathcal{N}(X,Y)^\ast\big)=1$, then $n(X^{\ast\ast})=n(Y^\ast)=1$.
\end{enumerate}
\end{Coro}

\section{Daugavet property and tensor products}\label{section:Daugavet}
In this section we study the relationship between the Daugavet property and tensor products. A sight to Corollary \ref{corollary-indextensorproduct=1} may lead to think that an analogous result can be true for the Daugavet property, that is, if $X\hat{\otimes}_\pi Y$ or $X\hat{\otimes}_\eps Y$ has the Daugavet property, do $X$ and $Y$ inherit this property? The answer is negative in general since, for instance, $L_1([0,1],Y)=L_1[0,1]\hat{\otimes}_\pi Y$ and $C([0,1],Y)=C[0,1]\hat{\otimes}_\eps Y$ have the Daugavet property for every Banach space $Y$, regardless that $Y$ has the Daugavet property or not. Our goal here is to show some cases in which the Daugavet property of a tensor product passes to one of the factors. To state our results, we need the definition and basic properties of the concept of slicely countably determined sets introduced in \cite{SCD}, where we refer for background. Let $A$ be a bounded subset of a Banach space $X$. A countable family $\{V_n\colon n\in\N\}$ of subsets of $A$ is called \emph{determining} for $A$ if the inclusion $A\subseteq\overline{\conv}(B)$ holds for every subset $B\subseteq A$ intersecting all the sets $V_n$. Recall that a \emph{slice} of $A$ is a nonempty intersection of $A$ with an open half space, and for $x^*\in X^\ast$ and $\delta>0$, we write
$$
\operatorname{Slice}(A,x^*,\delta):=\{x\in A\colon \re x^*(x)>\sup \re x^*(A)-\delta\}.
$$
The set $A$ is said to be \emph{slicely countably determined} (\emph{SCD} in short) if there exists a countable family of slices which is determining for $A$. Examples of SCD sets are the Radon-Nikod\'{y}m set and those sets not containing basic sequences equivalent to the basis of $\ell_1$ \cite{SCD}. A bounded linear operator $T\colon X\longrightarrow Y$ between two Banach spaces $X$ and $Y$ is an \emph{SCD-operator} if $T(B_X)$ is an SCD set, so examples of SCD-operators are the strong Radon-Nikod\'{y}m ones and those not fixing copies of $\ell_1$ \cite{SCD}. Finally, let us comment that a set $A$ is SCD if and only if $\overline{\conv}(A)$ is SCD \cite[Proposition~7.17]{SpearsBook}. Consequently, if $A$ is SCD then so is every set $C$ satisfying $A\subset \overline{C} \subset \overline{\conv}(A)$.

The main result of this section is the following one which deals with projective tensor products.

\begin{Theo} \label{theorem:Daugavet-tensorproduct}
Let $X$, $Y$ be Banach spaces. Suppose that $B_Y$ is an SCD set and $X\hat{\otimes}_\pi Y$ has the Daugavet property. Then, $X$ has the Daugavet property.
\end{Theo}

We need the following preliminary result which shows that the projective tensor product of an SCD-operator and a rank-one operator is again an SCD-operator on a projective tensor product.

\begin{Lemm}\label{lemma-SCD-proj}
	Let $X$, $Y$ be Banach spaces, let $S\in\mathcal{L}(X)$ be a rank-one operator and let $T\in\mathcal{L}(Y)$ be an SCD-operator. Then $S\otimes_\pi T \in \mathcal{L}(X\hat{\otimes}_\pi Y)$ is an SCD-operator.
\end{Lemm}

\begin{proof}
	We may and do assume that $\|S\|=\|T\|=1$. In order to prove that $[S\otimes_\pi T](B_{X\hat{\otimes}_\pi Y})$ is SCD it is enough to prove that $S(B_X)\otimes T(B_Y)$ is SCD as
	$$
	S(B_X)\otimes T(B_Y)\subset [S\otimes_\pi T](B_{X\hat{\otimes}_\pi Y})=[S\otimes_\pi T](\overline{\conv}(B_X\otimes B_Y))\subset\overline{\conv}\left(S(B_X)\otimes T(B_Y)\right).
	$$
	Since $S$ is a rank-one operator, there exist $x_0\in S_X$ and $\Gamma\subset \mathbb{K}$ such that $S(B_X)=\Gamma \{x_0\}$ ($\Gamma$ equals either $B_\mathbb{K}$ or its interior). So we can write $S(B_X) \otimes T(B_Y)=\{x_0\}\otimes\Gamma T(B_Y)$. Observe that $\Gamma T(B_Y)$ is SCD since $T(B_Y)$ is SCD and
	$$
	T(B_Y)\subset \overline{\Gamma T(B_Y)} \subset \overline{T(B_Y)}.	
	$$
	Therefore, for each $n\in \N$ we can find $V_n=\operatorname{Slice}(\Gamma T(B_Y),y_n^*,\eps_n)$ such that the sequence $\{V_n \colon n\in \N\}$ is determining for $\Gamma T(B_Y)$. Now fix $x_0^\ast\in S_{X^\ast}$ satisfying $\re x_0^\ast(x_0)=1$ and, for each $n\in \N$, define $\varphi_n\in (X\hat{\otimes}_\pi Y)^\ast=\mathcal{L}(X,Y^\ast)$ by $\varphi_n(x)=x_0^\ast(x) y_n^*$ for every $x\in X$. Let us prove that the slices
	$$
	S_n=\{x_0\}\otimes V_n=\operatorname{Slice}\left(\{x_0\}\otimes\Gamma T(B_Y),\varphi_n,\eps_n\right) \qquad (n\in\N)
	$$
	form a determining sequence for $\{x_0\}\otimes\Gamma T(B_Y)$. Indeed, if $B\subseteq\{x_0\}\otimes\Gamma T(B_Y)$ intersects all the $S_n$, then $B$ must be of the form $\{x_0\}\otimes B_2$ with $B_2\subset \Gamma T(B_Y)$ satisfying $B_2\cap V_n\neq\emptyset$ for every $n\in\N$. Since $V_n$ is determining for $\Gamma T(B_Y)$, this implies that $\Gamma T(B_Y)\subset\overline{\conv}(B_2)$ and thus
	$$
	\{x_0\}\otimes\Gamma T(B_Y)\subset\{x_0\}\otimes\overline{\conv}(B_2)\subset\overline{\conv}(B)
	$$
	which shows that the sequence $\{S_n\}$ is determining for $\{x_0\}\otimes\Gamma T(B_Y)=S(B_X)\otimes T(B_Y)$.
	\end{proof}
We are ready to show that the Daugavet property passes from the projective tensor product to one of the factors if the other one is SCD.

\begin{proof}[Proof of Theorem~\ref{theorem:Daugavet-tensorproduct}]
	Fix a rank-one operator $S\in\mathcal{L}(X)$ and consider $T=S\otimes_\pi \Id_Y\in\mathcal{L}(X\hat{\otimes}_\pi Y)$ which satisfies $\|T\|=\|S\|$ and is an SCD-operator by Lemma \ref{lemma-SCD-proj}. Since $X\hat{\otimes}_\pi Y$ has the Daugavet property, $T$ satifisfies the Daugavet equation by \cite[Corollary 5.9]{SCD}:
	$$
	\left\|\Id_{X\hat{\otimes}_\pi Y}+T\right\|=1+\|T\|=1+\|S\|.
	$$
	By the definition of $T$ we have
	$$
	\left\|\Id_{X\hat{\otimes}_\pi Y}+T\right\|=\left\|(\Id_X+S)\otimes_\pi \Id_Y\right\|=\left\|\Id_X+S\right\|
	$$
	and so $\|\Id_X+S\|=1+\|S\|$, as desired.
\end{proof}

We do not know whether the corresponding result for the injective tensor product is true or not. But we have the following positive result in the same line.

\begin{Prop}\label{prop:DPr-frechet-injective}
	Let $X$, $Y$ be Banach spaces such that $X\hat{\otimes}_\eps Y$ has the Daugavet property. Suppose that the norm of $Y$ is Fr\'{e}chet differentiable at a point $y_0\in S_Y$. Then, $X$ has the Daugavet property.
\end{Prop}

We need the following characterization of the Daugavet property which appears in the seminal paper \cite{KSSW}.

\begin{Lemm}[\mbox{\cite[Lemma~2.2]{KSSW}}]\label{lemma:Daugavet}
Let $X$ be a Banach space. Then the following assertions are equivalent:
\begin{itemize}
\item[(i)] $X$ has the Daugavet property;
\item[(ii)] for every $x\in S_X$, $x^*\in S_{X^\ast}$ and $\eps>0$, there is $y\in \operatorname{Slice}(S_X,x^*,\eps)$ such that $\|x+y\|>2-\eps$;
\item[(iii)] for every $x\in S_X$, $x^*\in S_{X^\ast}$ and $\eps>0$, there is $y^*\in \operatorname{Slice}(S_{X^\ast},x,\eps)$ such that $\|x^*+y^*\|>2-\eps$.
\end{itemize}
\end{Lemm}

\begin{proof}[Proof of Proposition~\ref{prop:DPr-frechet-injective}]
	Since the norm of $Y$ is Fr\'{e}chet differentiable at $y_0\in S_Y$, there is a unique $y_0^*\in S_{Y^\ast}$ which is strongly exposed in $B_{Y^\ast}$ by $y_0$, that is,
	\begin{equation}\label{strongly-exposed}
	\forall\eps>0 \ \exists\delta>0 \colon  y^*\in B_{Y^\ast}, \ \re y^*(y_0)>1-\delta \Longrightarrow \|y_0^*-y^*\|<\eps.
	\end{equation}
	Given $x_0^*\in S_{X^\ast}$ and $x_0\in B_X$, we consider $u_0=x_0\otimes y_0 \in B_{X\hat{\otimes}_\eps Y}$ and $\varphi_0=x_0^*\otimes y_0^*\in S_{(X\hat{\otimes}_\eps Y)^\ast}$. Since $X\hat{\otimes}_\eps Y$ has the Daugavet property, by Lemma~\ref{lemma:Daugavet}, fixed $\eps>0$, we may find $\varphi\in \operatorname{Slice}\left(B_{(X\hat{\otimes}_\eps Y)^\ast},u_0,\delta\right)$ such that $\|\varphi_0+\varphi\|>2-\eps$. As $B_{(X\hat{\otimes}_\eps Y)^\ast}= \overline{\conv}^{w^\ast}(B_{X^\ast}\otimes B_{Y^\ast})$, we may suppose that $\varphi=x^*\otimes y^*$ with $x^*\in B_{X^\ast}$ and $y^*\in B_{Y^\ast}$. On the one hand, from $\varphi\in \operatorname{Slice}\left(B_{(X\hat{\otimes}_\eps Y)^\ast},u_0,\delta\right)$ it follows that $x^*\in \operatorname{Slice}(B_{X^\ast},x_0,\delta)$ and $y^*\in \operatorname{Slice}(B_{Y^\ast},y_0,\delta)$. On the other hand, we can write
	$$
	2-\eps<\|\varphi_0+\varphi\|\leq\|x_0^*\otimes y_0^*+x^*\otimes y_0^*\| + \|x^*\otimes y_0^*-x^*\otimes y^*\|\leq\|x_0^*+x^*\|+\|y_0^*-y^*\|.
	$$
But $\|y_0^*-y^*\|<\eps$ by \eqref{strongly-exposed}, so we deduce that $\|x_0^*+x^*\|>2-2\eps$. Now, $X$ has the Daugavet property by Lemma~\ref{lemma:Daugavet}.
\end{proof}

We can obtain a result similar to the previous one for the projective tensor product which does not follow from Theorem~\ref{theorem:Daugavet-tensorproduct}.

\begin{Prop}\label{prop:DPr-frechet-projective}
	Let $X$, $Y$ be Banach spaces such that $X\hat{\otimes}_\pi Y$ has the Daugavet property. Suppose that the norm of $Y^\ast$ is Fr\'{e}chet differentiable at a point $y_0^*\in S_{Y^\ast}$. Then, $X$ has the Daugavet property.
\end{Prop}	

\begin{proof}
	Since $y_0^*\in S_{Y^\ast}$ is a point of Fr\'{e}chet differentiability, there is a unique $y_0\in S_{Y}$ satisfying:
	\begin{equation}\label{*strongly-exposed}
	\forall\eps>0 \ \exists\delta>0 \colon  y\in B_{Y}, \ \re y_0^*(y)>1-\delta \Longrightarrow \|y_0-y\|<\eps.
	\end{equation}
	Given $x_0\in S_{X}$ and $x_0^*\in B_{X^\ast}$, we consider $u_0=x_0\otimes y_0 \in S_{X\hat{\otimes}_\pi Y}$ and $\varphi_0=x_0^*\otimes y_0^*\in B_{(X\hat{\otimes}_\pi Y)^\ast}$. Since $X\hat{\otimes}_\pi Y$ has the Daugavet property and $B_{X\hat{\otimes}_\pi Y}=\overline{\conv}(B_X\otimes B_Y)$, fixed $\eps>0$, we may find $u\in \operatorname{Slice}\left(B_{X\hat{\otimes}_\pi Y},\varphi_0,\delta\right)$ of the form $u=x\otimes y$ with $x\in B_X$ and $y\in B_Y$ such that $\|u_0+u\|>2-\eps$. On the one hand, from $u\in \operatorname{Slice}\left(B_{X\hat{\otimes}_\pi Y},\varphi_0,\delta\right)$ it follows that $x\in \operatorname{Slice}(B_X,x_0^*,\delta)$ and $y\in \operatorname{Slice}(B_Y,y_0^*,\delta)$. On the other hand, we have
	$$
	2-\eps<\|u_0+u\|\leq\|x_0\otimes y_0+x\otimes y_0\|+\|x\otimes y-x\otimes y_0\|\leq\|x_0+x\|+\|y-y_0\|.
	$$
But $\|y-y_0\|<\eps$ by \eqref{*strongly-exposed}, so $\|x_0+x\|>2-2\eps$. Now, $X$ has the Daugavet property by Lemma~\ref{lemma:Daugavet}.
\end{proof}

\vspace*{0.5cm}

\noindent \textbf{Acknowledgement:\ } The authors thank Abraham Rueda Zoca for many conversations on the topic of this manuscript.

\end{document}